\documentclass[11pt]{amsart}  %@@
\usepackage{amssymb, eucal}
\usepackage[all,arc]{xy}
\usepackage{enumerate}
\usepackage{color}
\newenvironment{tfae}{
\begin{enumerate}}{\end{enumerate}}%\def\theenumi{\arabic{enumi}}

\newtheorem{prop}{Proposition}[section]
\newtheorem{lemma}[prop]{Lemma}
\newtheorem{theorem}[prop]{Theorem}
\newtheorem{corollary}[prop]{Corollary}
\theoremstyle{definition}
\newtheorem{definition}[prop]{Definition}

\newtheorem{example}[prop]{Example}
\newtheorem{examples}[prop]{Examples}
\newtheorem{remark}[prop]{Remark}

\renewcommand{\subsection}[1]{\addtocounter{subsection}{1}
\vspace*{2ex}\noindent\textbf{\thesubsection}\hspace{1ex}{\bf #1}}

\usepackage{chngcntr}
\counterwithin*{equation}{section}

\newdir{ >}{{}*!/-8pt/@{>}}

\def\mathrmdef#1{\expandafter\def\csname#1\endcsname{{\rm#1}}}
\mathrmdef{Aut}\mathrmdef{can}\mathrmdef{ev}\mathrmdef{Ker}\mathrmdef{id}\mathrmdef{Id}\mathrmdef{lex}\mathrmdef{max}\mathrmdef{op}\mathrmdef{pr}
\mathrmdef{prod}\mathrmdef{sym}\mathrmdef{Ult}\mathrmdef{Pt}\mathrmdef{End}\mathrmdef{cod}\mathrmdef{SplExt}

\def\mathsfdef#1{\expandafter\def\csname#1\endcsname{{\sf#1}}}
\mathsfdef{F}
\mathsfdef{P}
\mathsfdef{Gp}
\mathsfdef{oGp}
\mathsfdef{R}
\mathsfdef{U}

 \def\mathbfdef#1{\expandafter\def\csname#1\endcsname{{\rm\bf#1}}}

\mathbfdef{Alg} \mathbfdef{Bool} \mathbfdef{C}\mathbfdef{Cat} \mathbfdef{CLC}
\mathbfdef{Comp}\mathbfdef{Conn}\mathbfdef{CompHaus}\mathbfdef{ContLat}
\mathbfdef{Cont}\mathbfdef{Fam} \mathbfdef{Ps} \mathbfdef{Gph}\mathbfdef{Grp}
\mathbfdef{Haus}\mathbfdef{Inf}
\mathbfdef{Lat}\mathbfdef{LC}\mathbfdef{Met}\mathbfdef{MetGrp}\mathbfdef{Mon} \mathbfdef{MultiOrd}
\mathbfdef{Ord}\mathbfdef{RelAlg}\mathbfdef{ProbMet}\mathbfdef{ProbMetGrp}\mathbfdef{ProbUMet}\mathbfdef{PsTop}\mathbfdef{Rel}
\mathbfdef{Set} \mathbfdef{SGrp} \mathbfdef{Sup} \mathbfdef{Top}
\mathbfdef{OrdGrp} \mathbfdef{OrdAb}\mathbfdef{UMet} \mathbfdef{UMetGrp}
\mathbfdef{CMon}

\newcommand{\Cats}[1]{#1\text{-}\Cat}
\newcommand{\Gphs}[1]{#1\text{-}\Gph}
\newcommand{\Grps}[1]{#1\text{-}\Grp}
\newcommand{\Rels}[1]{#1\text{-}\Rel}
\newcommand{\VCat}{\Cats{V}}
\newcommand{\WCat}{\Cats{W}}
\newcommand{\VGph}{\Gphs{V}}
\newcommand{\VGrp}{\Grps{V}}
\newcommand{\WGrp}{\Grps{W}}
\newcommand{\VRel}{\Rels{V}}

\def\TT{\mathbb{T}}

\def\relto{{\longrightarrow\hspace*{-2.8ex}{\mapstochar}\hspace*{2.6ex}}}

\def\two{\mbox{\bf 2}}

\begin{document}  %@@
\title{On the categorical behaviour of $V$-groups}  %@@
\author{Maria Manuel Clementino}
\address{University of Coimbra, CMUC, Department of Mathematics, 3001-501 Coimbra, Portugal}\thanks{}
\email{mmc@mat.uc.pt}
\author{Andrea Montoli}
\address{Dipartimento di Matematica ``Federigo Enriques'', Universit\`{a} degli Studi di Milano, Via Saldini 50, 20133 Milano, Italy}
\email{andrea.montoli@unimi.it}
\thanks{}

\begin{abstract}
We consider compatible group structures on a $V$-category, where $V$ is a quantale, and we study the topological and algebraic properties of such groups. Examples of such structures are preordered groups, metric and ultrametric groups, probabilistic (ultra)metric groups. In particular, we show that, when $V$ is a frame, symmetric $V$-groups satisfy very strong categorical-algebraic properties, typical of the category of groups. In particular, symmetric $V$-groups form a protomodular category.
\end{abstract}
\maketitle  %@@

\section{Introduction}

In the paper \cite{CMFM19}, preordered groups have been studied from a categorical point of view. On one hand, using the analogies of topological nature between the categories \Ord\ of preordered sets and monotone maps and \Top\ of topological spaces and continuous maps, one can describe limits and colimits in the category \OrdGrp\ of preordered groups using the properties of the forgetful functors to \Ord\ and to the category \Grp\ of groups. On the other hand, the main difference between topological groups and preordered ones is that the former are internal groups in \Top\, while the latter aren't, since it is not required that the inversion map is monotone. For these reasons, the algebraic properties of \OrdGrp\ are not so good as the ones of topological groups (for instance, the Split Short Five Lemma does not hold in $\OrdGrp$). In order to understand better the algebraic behaviour of $\OrdGrp$, the strategy used in \cite{CMFM19} was mainly based on the well known fact that a compatible preorder on a group is completely determined by the submonoid of positive elements. Using this observation, it was shown in \cite{CMFM19} that \OrdGrp\ shares many categorical-algebraic properties with the category \Mon\ of monoids. In particular, in both categories it is possible to identify intrinsically a full subcategory, of so-called \emph{protomodular objects} \cite{MRVdL objects}, which has basically the same algebraic properties of $\Grp$. In the case of monoids, this subcategory is precisely $\Grp$, while in \OrdGrp\ it is the subcategory whose objects are the groups equipped with a compatible preorder which is symmetric, i.e. a congruence. These are precisely the internal groups in $\Ord$. \\

Another approach is possible. Indeed, it is known \cite{Law73} that preordered sets can be seen as categories enriched in the lattice $\two = \{ \bot, \top \}$, with $\bot < \top$. Following this point of view, preordered groups can be seen as those monoid objects in the monoidal category $\two$-$\Cat$ (of categories enriched in $\two$) that are groups. In \cite{Law73} several other examples of categories enriched in a \emph{quantale} (i.e. in a complete lattice equipped with a tensor product which is distributive w.r.t. arbitrary joins) are considered, allowing to describe in this way important mathematical structures, like metric spaces. \\

The aim of the present paper is to follow this alternative approach, based on $V$-categories, where $V$ is a quantale satisfying suitable properties, to study in a common framework structures like preordered groups and metric groups. We consider what we call $V$-groups, namely monoid objects in the monoidal category $\VCat$, of $V$-categories and $V$-functors, that are groups. The advantage of working in this setting is twofold. On one hand it allows to extend to a wide class of situations the results obtained for preordered groups in \cite{CMFM19}. Actually, some of these results appear even more interesting in some of the new examples. For instance, the ``good'' objects in \OrdGrp\ are those whose preorder is symmetric, and this somehow destroys the ordered structure of the object (in particular, if we restrict our attention to partially ordered groups, the only good ones are the discrete groups), while requiring symmetry for metric groups is much more classical and natural. On the other hand, the proofs we get following the $V$-categorical approach are significantly simpler than the ones we had in \cite{CMFM19} for preordered groups, using the positive cone. \\

After recalling the necessary background on $V$-categories, we obtain relevant information on limits, colimits and factorization systems in the category $\VGrp$ of $V$-groups using the properties of the forgetful functors into $\VCat$ and $\Grp$, as well as some topological properties of $V$-groups. Then we will concentrate our attention on the algebraic properties of $V$-groups, observing that they are well-behaved especially when $V$ is a frame, which means that the tensor product defined on $V$ coincides with the meet operation. Under this assumption, we show that the protomodular objects are precisely the symmetric $V$-groups, and hence the full subcategory of $\VGrp$ whose objects are the symmetric $V$-groups is protomodular. Furthermore, observing that symmetric $V$-groups are precisely the internal groups in the cartesian closed category $\VCat$, we show that the full subcategory $\VGrp_\sym$ of symmetric $V$-groups has representable actions (in the sense of \cite{BJK action representative}) and is locally algebraically cartesian closed (in the sense of \cite{Gray}). Finally, we study split extensions in $\VGrp$, showing that all the compatible $V$-category structures on the semidirect product of two $V$-groups are intermediate between the one given by the tensor and a \emph{lexicographic} one, a generalization of the lexicographic order on a product.

\section{$V$-categories} \label{Section V-categories}

Let $V$ be a \emph{commutative and unital quantale}, that is, $V$ is a complete lattice (with top element $\top$ and bottom $\bot$) equipped with a symmetric and associative tensor product $\otimes$, with unit $k$, which preserves joins, that is,
\[v\otimes \bigvee_i u_i=\bigvee_i(v\otimes u_i),\mbox{  and  } v\otimes \bot=\bot\]
for every $v\in V$, and family $(u_i)_{i\in I}$ in $V$.
Therefore it has a right adjoint, $\hom$; that is, for all $u\in V$, $(\;\;)\otimes u\colon V\to V$ is left adjoint to $\hom\colon V\to V$, or, equivalently, for every $v,w\in V$, \[v\otimes u\leq w \iff v\leq\hom(u,w).\]
Moreover, we also assume that arbitrary joins distribute over finite meets, that is, as a lattice, $V$ is a frame.

\begin{example}
Any frame $V$ defines a commutative and unital quantale, with $\otimes=\wedge$ and $k=\top$. This type of quantales will be used often, and we will refer to them saying that \emph{the quantale $V$ is a frame}. (It should not be confused with our additional assumption that $V$, as a lattice, is a frame.)
\end{example}

In order to define a $V$-category, we will make use of the bicategory $\VRel$, whose objects are sets and whose morphisms $r\colon X\relto Y$ are \emph{$V$-relations}, i.e. maps $X\times Y\to V$; the composition of $V$-relations $r\colon X\relto Y$, $s\colon Y\relto Z$ is a $V$-relation $X\relto Z$ defined by relational composition:
\[(s\cdot r)(x,z)=\bigvee_{y\in Y}r(x,y)\otimes s(y,z).\]
The identity morphism $X\relto X$ is given by
\[1_X(x,x')=\left\{\begin{array}{ll}
k&\mbox{ if }x=x',\\
\bot&\mbox{ elsewhere.}
\end{array}\right.\]
In general every map $f\colon X\to Y$ can be considered as the $V$-relation
\[f(x,y)=\left\{\begin{array}{ll}
k&\mbox{ if }y=f(x),\\
\bot&\mbox{ elsewhere,}
\end{array}\right.\]
so there is a (non-full, bijective on objects) inclusion $\Set\to\VRel$.

Given $r,r'\colon X\relto Y$, $r\leq r'$ if, for all $x\in X$, $y\in Y$, $r(x,y)\leq r'(x,y)$ in $V$.

A \emph{$V$-category} is a pair $(X,a)$, where $X$ is a set and $a\colon X\relto X$ is a $V$-relation such that
\[1_X\leq a\mbox{ and }a\cdot a\leq a;
\]
in pointwise notation this means that:
\begin{description}
\item[\rm (R)] $(\forall \, x\in X)\;\;\;k\leq a(x,x)$,
\item[\rm (T)] $(\forall \, x,x',x''\in X)\;\;\; a(x,x')\otimes a(x',x'')\leq a(x,x'')$;
\end{description}
Property (R) is usually called \emph{reflexivity} while (T) is \emph{transitivity}; they are also called, respectively, \emph{unit} and \emph{associativity} axioms. Pairs $(X,a)$ satisfying (R) are called \emph{$V$-graphs}.

Given two $V$-categories (or $V$-graphs) $(X,a)$ and $(Y,b)$, a map $f\colon X\to Y$ is a \emph{$V$-functor} $f\colon(X,a)\to(Y,b)$ if $f\cdot a\leq b\cdot f$; in pointwise notation this means
\[(\forall x,x'\in X)\;\;\;a(x,x')\leq b(f(x),f(x')).\]
It is easy to check that $V$-categories and $V$-functors define a category, denoted by $\VCat$; the category of $V$-graphs and $V$-functors is denoted by $\VGph$.

\begin{remark}
For each $V$-relation $r\colon X\relto Y$ we can define the opposite relation $r^\circ\colon Y\relto X$ by $r^\circ(y,x)=r(x,y)$. Given a $V$-category $(X,a)$, $(X,a^\circ)$ is also a $V$-category, usually called the \emph{dual} of $(X,a)$. Based on the lemma below, we can conclude that this assignment defines a functor $(\;\;)^\op\colon\VCat\to\VCat$, with $(X,a)^\op=(X,a^\circ)$ and $f^\op=f$.

When $a=a^\circ$ we call the $V$-category $(X,a)$ \emph{symmetric}. The full subcategory of $\VCat$ of symmetric $V$-categories will be denoted by $\VCat_\sym$.
\end{remark}

The proof of the following Lemma is straightforward, and can be found in \cite{CH03}.

\begin{lemma}
\begin{enumerate}
\item If we consider the map $f\colon X\to Y$ as a $V$-relation, we have that
\[f\cdot f^\circ\leq 1_Y\mbox{ and }1_X\leq f^\circ\cdot f;\]
that is, $f^\circ$ is the right adjoint of $f$.
\item For $V$-categories $(X,a),(Y,b)$ and a map $f\colon X\to Y$, the following conditions are equivalent to $f\colon(X,a)\to(Y,b)$ being a $V$-functor:
\begin{tfae}
\item $a\leq f^\circ\cdot b\cdot f$;
\item $f\cdot a\cdot f^\circ\leq b$;
\item $f\cdot a^\circ\leq b^\circ\cdot f$.
\end{tfae}
\end{enumerate}
\end{lemma}

Consider the following commutative diagram
\[\xymatrix{\VCat_\sym\ar[r]^-{\mathsf{I}_2}\ar[dr]&\VCat\ar[r]^-{\mathsf{I}_1}\ar[d]&\VGph\ar[ld]\\
&\Set,}\]
where the horizontal functors are embeddings and the vertical ones are forgetful functors.

\begin{prop}\label{refl}
\begin{enumerate}
\item The functor $\mathsf{I}_1$ is a right adjoint, that is, \VCat\ is a reflective subcategory of \VGph.
\item The functor $\mathsf{I}_2$ has both a left and a right adjoint, that is, $\VCat_\sym$ is both a reflective and a coreflective subcategory of \VCat.
\item The forgetful functors $\VGph\to\Set$, $\VCat\to\Set$, and $\VCat_\sym\to\Set$ are topological.
\end{enumerate}
\end{prop}

\begin{proof}
(1) The left adjoint is built by iteration of the pointed endofunctor $\VGph\to\VGph$ that assigns to each $(X,a)$ the $V$-graph $(X,a\cdot a)$ (see \cite[Theorem 4.4]{CH03} for details).

(2) The symmetrization of a $V$-category $(X,a)$ is defined by $\hat{a}(x,x')=a(x,x')\wedge a(x',x)$, for every $x,x'\in X$. It is easily checked that this defines the right adjoint to $\mathsf{I}_2$. The left adjoint of $\mathsf{I}_2$ is built in two steps: first define $\tilde{a}(x,x')=a(x,x')\vee a(x',x)$, and then $\check{a}$ is obtained by iteration of the pointed endofunctor $\VGph\to\VGph$ of (1), applied to $\tilde{a}$. It is clear that the symmetry of $\check{a}$ follows from the symmetry of $\tilde{a}$.

To show (3), using (1) and (2) it is enough to prove that $\VGph\to\Set$ is topological (see \cite{AHS}). This follows the arguments of \cite[Theorem 4.5]{CH03}. Given $f_i\colon X\to (Y_i,b_i)$, the largest possible $V$-graph structure on $X$ that makes all the $f_i$ $V$-functors is
\[a:=\bigwedge_{i\in I} f_i^\circ\cdot b\cdot f_i,\]
and it is easy to check that $a$ verifies (R).
\end{proof}

The monoidal structure of $V$ induces a monoidal structure on $\VCat$; indeed, for $V$-categories $(X,a)$ and $(Y,b)$, we define $(X,a)\otimes (Y,b)$ by $(X\times Y,a\otimes b)$, where $(a\otimes b)((x,y),(x',y'))=a(x,x')\otimes b(y,y')$, with $f\otimes g=f\times g$. The unit is the $V$-category $I=(\{*\}, \kappa)$, where $\kappa(*,*)=k$.

\begin{theorem}
$\VCat$ is a monoidal closed category.
\end{theorem}

\begin{proof}
It is straightforward to prove that $(\;)\otimes (X,a)\colon\VCat\to\VCat$ is left adjoint to $[(X,a),(\;)]\colon\VCat\to\VCat$, where
$[(X,a),(Y,b)]=(\{f\colon(X,a)\to(Y,b)$ $V$-functor$\},[\;,\;])$,
with
\[ [f,g]=\bigwedge_{x\in X} b(f(x),g(x)),\]
for every pair of $V$-functors $f,g\colon(X,a)\to(Y,b)$. (See \cite{Law73} for details.)
\end{proof}

\begin{remark}
It is clear that, with the same construction, $\VGph$ and $\VCat_\sym$ are also monoidal closed categories.
\end{remark}

\begin{remark}\label{re:change}
As shown in \cite[Section 3.5]{HST14a}, every lax homomorphism $\varphi\colon V\to W$ of quantales, so that $\varphi$ is order preserving, $\varphi(u)\otimes\varphi(v)\leq\varphi(u\otimes v)$, and $l\leq\varphi(k)$, where $u,v\in V$ and $k, l$ are the units of $V$ and $W$ respectively, induces a functor $B_\varphi\colon\VCat\to\WCat$, with $B_\varphi(X,a)=(X,\varphi\cdot a)$, and $B_\varphi(f)=f$. Moreover, any adjunction $\psi\dashv \varphi$ of lax homomorphisms of quantales induces an adjunction $B_\psi\dashv B_\varphi$.

In particular, for every non-degenerate quantale $V$, we may define two lax homomorphisms $\iota,\tau\colon\two\to V$, with $\iota(\bot)=\tau(\bot)=\bot$, $\iota(\top)=k$, and $\tau(\top)=\top$ (which obviously coincide when the quantale is \emph{integral}, that is $k=\top$); $\iota$ has always a right adjoint, the so called \emph{pessimist's map} $p\colon V\to\two$, with $p(v)=\top\iff v\geq k$. The \emph{optimist's map} $o\colon V\to\two$, defined by $o(v)=\top\iff v\neq\bot$, is a lax homomorphism if, and only if, for any $u,v\in V$,
\[u\otimes v=\bot\implies u=\bot \mbox{ or } v=\bot.\]
We call these quantales \emph{optimistic}. For optimistic quantales the optimist's map $o$ is left adjoint to $\tau$. Therefore, \emph{when $V$ is integral and optimistic, we have a chain of adjunctions}
\[\xymatrix{\two\ar[rr]|{\,\iota\,}&&V\ar@/_1pc/[ll]_o\ar@/^1pc/[ll]^p\ar@{}[ll]^{\bot}\ar@{}[ll]_{\bot}}.\]
\end{remark}

\begin{examples}\label{ex:vcats}
\begin{enumerate}
\item If $V=\two=(\{\bot,\top\},\leq)$ with $\otimes=\wedge$, then $\Cats{\two}$ is the category \Ord\ of preordered sets and monotone maps.
\item When $V=P_+=([0,\infty],\geq)$ is the complete real half-line as studied by Lawvere in \cite{Law73}, with $\otimes=+$, and then $\hom(u,v)=v\ominus u=\max(v-u,0)$, $\Cats{P_+}$ is the category \Met\ of \emph{Lawvere (generalized) metric spaces} and non-expansive maps.
    Since $P_+$ is both integral and optimistic, thanks to Remark \ref{re:change} \Ord\ embeds in \Met\ both reflectively and coreflectively:
   \[\xymatrix{\Ord\ar[rr]|{\,B_\iota \,}&&\Met\ar@/_1pc/[ll]_{B_o}\ar@/^1pc/[ll]^{B_p}\ar@{}[ll]^{\bot}\ar@{}[ll]_{\bot}}.\]

    If we take instead $P_\max$, so that in $([0,\infty],\geq)$ we take $\otimes=\wedge$ (note that this is $\max$ for the usual order in the real numbers), then $\Cats{P_\max}$ is the category \UMet\ of \emph{ultrametric spaces} and non-expansive maps. The identity $P_\max\to P_+$ is a lax homomorphism, inducing an embedding $\UMet\to\Met$.

   \item The unit interval $[0,1]$, with its usual order, is a complete lattice. It can be equipped with different tensor products, making $([0,1],\leq)$ a quantale. Here we mention the most relevant ones: the minimum $\wedge$, the multiplication $\times$, and the \emph{\L{}ukasiewicz tensor} $\oplus$, defined by $u\oplus v= \max(u+v-1,0)$, the three of them with unit $1$. We will denote the corresponding quantales respectively by $[0,1]_\wedge$, $[0,1]_\times$, and $[0,1]_\oplus$.

       The bijection $\varphi\colon [0,1]\to [0,\infty]$, with $u\mapsto -\ln u$, defines an isomorphism of quantales $[0,1]_\times\to P_+$, and therefore the category $\Cats{[0,1]_\times}$ is isomorphic to $\Cats{P_+}$, i.e., the category \Met\ of generalized metric spaces.

       The same map $\varphi$ induces an isomorphism of quantales $[0,1]_\wedge\to P_\max$, henceforth $\Cats{[0,1]_\wedge}$ is isomorphic to the category \UMet\ of ultrametric spaces.

       The category $\Cats{[0,1]_\oplus}$ is worth to be mentioned: it is isomorphic to the category $\Met_1$ of \emph{(generalized) metric spaces bounded by $1$} and non-expansive maps. Indeed, $\psi\colon[0,1]_\oplus\to P_+$, with $\psi(u)=1-u$ for every $u\in [0,1]$, is a lax homomorphism of quantales, defining a functor $B_\psi\colon\Cats{[0,1]_\oplus}\to\Met$. It is easily checked that it is an embedding, with image $\Met_1$. The isomorphism $\Cats{[0,1]_\oplus}\to \Met_1$ assigns to each $[0,1]_\oplus$-category $(X,a)$ the metric space $(X,\overline{a})$, with $\overline{a}(x,x')=1-a(x,x')$ for every $x,x'\in X$, and keeps morphisms unchanged.
       (For more information on tensor products on $([0,1],\leq)$ see for instance \cite{CH17} and the references there.)

\item When $V$ is the quantale $\Delta$ of \emph{distribution functions}, that is,
\[\Delta=\{\varphi\colon[0,\infty]\to[0,1]\,|\,\varphi\mbox{ is monotone and }\varphi(x)=\bigvee_{y<x}\varphi(y)\}\]
with the pointwise order and $\otimes$ given by
\[(\varphi\otimes\psi)(x)=\bigvee_{y+z\leq x}\,\varphi(y)\times\psi(z),\]
then $\Cats{\Delta}$ is the category $\ProbMet$ of \emph{probabilistic metric spaces} and non-expansive maps, as studied in \cite{HR13} (see also \cite{CH17}).

We note that, as observed in \cite[Section 3.3]{HR13}, the natural embedding $P_+\to\Delta$ has both a left and a right adjoint, and so we have the following chain of adjunctions
\begin{equation}\label{eq:adjs}
\xymatrix{\Ord\ar[rr]|{\,B_\iota \,}&&\Met\ar@/_1pc/[ll]_{B_o}\ar@/^1pc/[ll]^{B_p}\ar@{}[ll]^{\bot}\ar@{}[ll]_{\bot}\ar[rr]&&
\ProbMet\ar@/_1pc/[ll]\ar@/^1pc/[ll]\ar@{}[ll]^-{\bot}\ar@{}[ll]_-{\bot}};
\end{equation}
in particular, $\Met$ embeds both reflectively and coreflectively in $\ProbMet$.

If we take $\Delta_\wedge$, that is, $\Delta$ together with the tensor product $\wedge$ defined pointwise, then $\Cats{\Delta_\wedge}$ is the category $\ProbUMet$ of \emph{probabilistic ultrametric spaces} and non-expansive maps, where $\UMet$ embeds both reflectively and coreflectively.
\end{enumerate}
\end{examples}

For more examples and information on $V$-categories we refer to \cite[Appendix]{HN20}, and \cite{CH17, HR13}.

\section{Basic results on $V$-groups}

A $V$-category $(X,a)$ equipped with a group structure $(X,+\colon X\times X\to X,i\colon X\to X,0\colon I \to X)$ such that $+\colon(X,a)\otimes (X,a)\to (X,a)$ is a $V$-functor is said to be a \emph{$V$-group}. Note that $0\colon(I,\kappa)\to (X,a)$, as every map from $(I,\kappa)$ to $(X,a)$, is a $V$-functor, and that we do not impose that the inversion $(X,a)\to(X,a)$ is a $V$-functor. We will use the additive notation although our groups need not be abelian. Given two $V$-groups $(X,a)$, $(Y,b)$, a \emph{$V$-homomorphism} $f\colon(X,a)\to(Y,b)$ is a $V$-functor which is a group homomorphism. We will denote by $\VGrp$ the category of $V$-groups and $V$-homomorphisms. We observe that \emph{a $V$-group is precisely a monoid object in the monoidal category $(\VCat, \otimes)$ which is a group}.

\begin{prop}\label{p:categ}
Let $(X,+)$ be a group and $(X,a)$ be a $V$-graph. The following conditions are equivalent:
\begin{tfae}
\item $+\colon(X,a)\otimes(X,a)\to(X,a)$ is a $V$-functor;
\item $(X,a)$ is a $V$-category and $a$ is \emph{invariant by shifting}, that is,
\begin{equation}\label{shift}
(\forall x,x',x''\in X)\;\;a(x',x'')=a(x'+x,x''+x).
\end{equation}
\end{tfae}
\end{prop}

\begin{proof}
If $+\colon(X,a)\otimes (X,a)\to (X,a)$ is a $V$-functor, then, for every $x\in X$, the maps $x+(\;)\colon (X,a)\to (X,a)$ and $(\;)+x\colon(X,a)\to (X,a)$, defined by
\[\xymatrix{X\ar[r]^-{\cong}& I\otimes X\ar[r]^{x\otimes \id_X}&X\otimes X\ar[r]^-{+}&X}\mbox{ and } \xymatrix{X\ar[r]^-{\cong}& X\otimes I\ar[r]^{\id_X\otimes x}&X\otimes X\ar[r]^-{+}&X},\] are $V$-functors; hence,
\[a(x',x'')\leq a(x'+x,x''+x)\leq a(x'+x-x,x''+x-x)=a(x',x'');\]
moreover,
\[a(x,x')\otimes a(x',x'')=a(x,x')\otimes a(0, -x'+x'')\leq a(x+0, x'-x'+x'')=a(x,x'').\]

Conversely, if \eqref{shift} holds, then, for every $x_1,x_2,x_1',x_2'\in X$,
\begin{align*}
a(x_1,x_2)\otimes a(x_1',x_2')&=a(0,-x_1+x_2)\otimes a(-x_1+x_2,-x_1+x_2+x_2'-x_1')\\
&\leq a(0,-x_1+x_2+x_2'-x_1')=a(x_1+x_1',x_2+x_2').
\end{align*}
\end{proof}

\begin{corollary}\label{cor:adjs}
Any lax homomorphism $\varphi\colon V\to W$ of quantales induces a functor $G_\varphi\colon\VGrp\to\WGrp$. Moreover, any adjunction $\psi\dashv \varphi$ induces and adjunction $G_\psi\dashv G_\varphi$.
\end{corollary}

\begin{proof}
It follows from Proposition \ref{p:categ}, since with $a$ also $\varphi\cdot a$ is invariant by shifting.
\end{proof}

\begin{corollary}
If $(X,a,+)$ and $(Y,b,+)$ are $V$-groups and $f:X\to Y$ is a homomorphism, then $f$ is a $V$-functor if, and only if,
\[(\forall \, x\in X)\;\; a(0,x)\leq b(0,f(x)).\]
\end{corollary}

\begin{proof}
For each $x,x'\in X$,
\[a(x,x')=a(0,x'-x)\leq b(0,f(x'-x))=b(0,f(x')-f(x))=b(f(x),f(x')).\]
\end{proof}

\begin{remark} \label{inversion V-functor iff symmetric}
\begin{enumerate}
\item If $(X,a)$ is a $V$-group, then $(X,a^\circ)$ is also a $V$-group. From \eqref{shift} it follows that the inversion is an isomorphism of $V$-groups $i\colon(X,a)\to(X,a^\circ)$, since
     \[a(x,y)=a(-x+x-y,-x+y-y)=a(-y,-x)=a^\circ(-x,-y).\]
     (We thank Dirk Hofmann for this observation.) Moreover, \emph{the inversion $i\colon(X,a)\to(X,a)$ is a $V$-functor if, and only if, the $V$-category $(X,a)$ is symmetric}, that is, if $(X,a,+)$ and $(X,a^\circ, +)$ coincide. A $V$-group $(X,a)$ with $i$ a $V$-functor will be called a \emph{symmetric $V$-group}. When $\otimes=\wedge$, a symmetric $V$-group is exactly an internal group in $\VCat$. We will denote by $\VGrp_\sym$ the full subcategory of \VGrp\ consisting of the symmetric $V$-groups.
\item \emph{When $\otimes=\wedge$, any finite group $(X,a,+)$ is a symmetric $V$-category.} Indeed, if $x\neq 0$, then $-x=nx$ for some $n$, and so \[a(0,-x)=a(0,nx)\geq a(0,x)\otimes \dots\otimes a(0,x)=a(0,x).\]
\end{enumerate}
\end{remark}

\begin{prop}\label{p:sym}
$\VGrp_\sym$ is both a reflective and a coreflective subcategory of \VGrp.
\end{prop}

\begin{proof}
Using Proposition \ref{p:categ}, it is enough to show that, for every $V$-group $(X,a,+)$, $\hat{a}$ and $\check{a}$, defined in the proof of Proposition \ref{refl}, are invariant by shifting, which follows from the fact that invariance by shifting is preserved by meets, joins, and composition of $V$-relations.
\end{proof}

\begin{examples}For each quantale $V$ described in Examples \ref{ex:vcats}, we can now consider the corresponding category $\VGrp$. Namely,
\begin{enumerate}
\item when $V=\two$, $\Grps{\two}$ is the category \OrdGrp\ of preordered groups and monotone group homomorphisms studied in \cite{CMFM19};

\item $\Grps{P_+}$ is the category \MetGrp\ whose objects are the (generalized) metric groups, i.e. the Lawvere generalized metric spaces with a group structure which is a non-expansive map, and whose arrows are the non-expansive group homomorphisms;

\item $\Grps{P_\max}$ is the category of (generalized) ultrametric groups and non-expansive group homomorphisms;

\item when $V=\Delta$ (resp. $V=\Delta_\wedge$), $\VGrp$ is the category \ProbMetGrp\ of probabilistic metric (resp. ultrametric) groups.
\end{enumerate}
Thanks to Corollary \ref{cor:adjs}, when $V$ is an integral and optimistic quantale, \OrdGrp\ embeds both reflectively and coreflectively in $\Grps{V}$:
\[\xymatrix{\OrdGrp\ar[rr]|{\,G_\iota\,}&&\Grps{V}\ar@/_1pc/[ll]_{G_o}\ar@/^1pc/[ll]^{G_p}\ar@{}[ll]^{\bot}\ar@{}[ll]_{\bot}};\]
and, moreover, all the embeddings of categories of $V$-categories we mentioned in Examples \ref{ex:vcats} restrict to embeddings of \VGrp. Namely, \eqref{eq:adjs} gives rise to the following chain of adjunctions
\[\xymatrix{\OrdGrp\ar[rr]|{\,G_\iota \,}&&\MetGrp\ar@/_1pc/[ll]_{G_o}\ar@/^1pc/[ll]^{G_p}\ar@{}[ll]^{\bot}\ar@{}[ll]_{\bot}\ar[rr]&&
\ProbMetGrp\ar@/_1pc/[ll]\ar@/^1pc/[ll]\ar@{}[ll]^-{\bot}\ar@{}[ll]_-{\bot}}.\]
\end{examples}

\begin{prop}\label{final}
If $(X,a,+)$ is a $V$-group, $Y$ is a group, and $f\colon X\to Y$ is a surjective group homomorphism, then $b:=f\cdot a\cdot f^\circ$ makes $(Y,b,+)$ a $V$-group and $f$ a $V$-homomorphism.
\end{prop}

\begin{proof}
Using Proposition \ref{p:categ}, it is enough to show that $(Y,b)$ is a $V$-category and $b$ satisfies \eqref{shift}. Note that, for every $y_1,y_2\in Y$,
\[b(y_1,y_2)=\bigvee_{x,x'\in X} f^\circ(y_1,x)\otimes a(x,x')\otimes f(x',y_2)=\bigvee_{f(x_i)=y_i} a(x_1,x_2).\]
For each $y, y_1,y_2,y_3\in Y$,
\begin{itemize}
\item $\displaystyle b(y,y)=\bigvee_{f(x)=f(x')=y} a(x,x')\geq \bigvee_{f(x)=y} a(x,x)\geq k$;
\item if $f(x)=y$, then
\[b(y_1,y_2)=\bigvee_{f(x_i)=y_i} a(x_1,x_2)=\bigvee_{f(x_i)=y_i} a(x_1+x,x_2+x)=\bigvee_{f(x_i')=y_i+y} a(x_1',x_2')=b(y_1+y,y_2+y);\]
\item and
\begin{align*}
\displaystyle b(y_1,y_2)\otimes b(y_2,y_3)=&\displaystyle\bigvee_{f(x_i)=y_i} a(x_1,x_2)\otimes\bigvee_{f(x_i')=y_i}a(x_2',x_3')=\bigvee_{f(x_i)=y_i=f(x_i')}a(x_1,x_2)\otimes a(x_2',x_3')\\
=&\displaystyle\bigvee_{f(x_i)=y_i=f(x_i')} a(x_1-x_2+x_2',x_2')\otimes a(x_2',x_3')\\
\leq&\displaystyle \bigvee_{f(x_i)=y_i=f(x_i')} a(x_1-x_2+x_2',x_3')=b(y_1,y_3).
\end{align*}
\end{itemize}
\end{proof}

We point out that the structure $b$ defined above is the least $V$-category structure making $f$ a $V$-functor. These $V$-homomorphisms are exactly the extremal epimorphisms in $\VGrp$, and the structure $b$ is the final structure with respect to the topological functor $\VGrp\to\Grp$ we study next.

\section{The category $\VGrp$}

It is well-known that the forgetful functor $|\;\;|\colon\Grp\to\Set$ is monadic, while, as we have already shown, the forgetful functor $|\;\;|\colon\VCat\to\Set$ is topological. Next we will check that these properties transfer to the forgetful functors $\U_1\colon\VGrp\to\Grp$ and $\U_2\colon\VGrp\to\VCat$, with $\U_1$ forgetting the V-categorical structure and $\U_2$ the group structure.

\begin{theorem}\label{top}
The functor $\U_1$ is topological, and the functor $\U_2\colon\VGrp\to\VCat$ is monadic.
Therefore we have the following commutative diagram
\[\xymatrix{&\VGrp\ar[ld]_{\mbox{(topological) }\U_1\;}\ar[rd]^{\;\U_2 \mbox{ (monadic)}}\\
\Grp\ar[dr]_{\mbox{(monadic) } |\;\;|\;}&&\VCat\ar[dl]^{\;|\;\;|\mbox{ (topological)}}\\
&\Set}\]
\end{theorem}

\begin{proof}
To show that $\U_1$ is a topological functor, let $(f_i\colon X\to (X_i,a_i))_{i\in I}$ be a family of group homomorphisms, with $(X_i,a_i)$, $i\in I$, $V$-groups. Then the initial structure on $\VCat$
\begin{equation}\label{initial}
a(x,y)=\bigwedge_{i\in I}a_i(f_i(x),f_i(y))
\end{equation} makes $X$ a $V$-group, and this defines clearly the $\U_1$-initial lifting for $(f_i)$. Indeed, to check that $(X,a,+)$ is a $V$-group it is enough to verify that \eqref{shift} holds: for every $x,y,z\in X$,
\[a(x+z,y+z)=\bigwedge_{i\in I}a_i(f_i(x)+f_i(z),f_i(y)+f_i(z))=\bigwedge_{i\in I}a_i(f_i(x),f_i(y))=a(x,y).\]
Topologicity of $\U_1$ gives that, with \Grp, also $\VGrp$ is complete and cocomplete.\\

To show that $\U_2$ is monadic, we will use \cite[Theorem 2.4]{VAtlantis}.

\noindent (a) \emph{$\U_2\colon\VGrp\to\VCat$ has a left adjoint}: Given a $V$-category $(X,a)$, let $\F X$ be the free group generated by $X$, and $\eta_X\colon X\to \F X$ the insertion of generators.
Then, for each $V$-group $(Y,b,+)$ and each $V$-functor $f\colon(X,a)\to \U_2(Y,b,+)$, there is a homomorphism $\overline{f}\colon\F X\to Y$ such that $\overline{f}\cdot\eta_X=f$. Equipping $\F X$ with the initial $V$-category structure $\hat{a}$ with respect to all the $\overline{f}\colon\F X\to(Y,b,+)$ as defined in \eqref{initial}, the inclusion $\eta_X\colon(X,a)\to\U_2(\F X,\hat{a})$ is a $V$-functor and it is universal from $X$ to $\U_2$, therefore $\U_2$ has a left adjoint as claimed.

\noindent (b) \emph{$\U_2$ reflects isomorphisms}: given a morphism $f\colon (X,a,+)\to (Y,b,+)$ in $\VGrp$, if $\U_2(f)$ is an isomorphism in $\VCat$ then $f$ is a bijective homomorphism and, for every $x,x'\in X$, $a(x,x')=b(f(x),f(x'))$. Therefore its inverse map is both a homomorphism and a $V$-functor, and so $f$ is an isomorphism in $\VGrp$.

\noindent (c) \emph{$\VGrp$ has and $\U_2$ preserves coequalizers of all $\U_2$-contractible coequalizer pairs.} First recall that the functor $|\;\;|\colon \Grp\to\Set$ is monadic. Given morphisms $f,g\colon X\to Y$ in $\VGrp$ such that $\U_2(f),\U_2(g)$ is a contractible pair in $\VCat$, we know that the coequalizer $q\colon Y\to Q$ in $\VGrp$ is preserved by $\U_1$, and so it is also preserved by $|\;\;|\cdot \U_1$, since $|\;\;|$ is monadic and $|\U_1(f)|, |\U_1(g)|$ form a contractible pair in \Set. Therefore $\U_2(q)$, as a split epimorphism that coequalizes $|\U_1(f)|, |\U_1(g)|$ in \Set, is the coequalizer of $\U_2(f),\U_2(g)$ in $\VCat$.
\end{proof}

We collect below properties of $\VGrp$ that follow from this proposition.

\begin{remark}\label{functors}
\begin{enumerate}
\item \emph{The functor $\U_1\colon\VGrp\to\Grp$ has both a left and a right adjoint.} The left adjoint ${\mathsf{L}}_1
\colon \Grp \to \VGrp$ equips a group $X$ with the discrete
V-category structure: $a(x,y)=k$ if $x=y$, and $a(x,y)=\bot$ otherwise. The right adjoint $\R_1
\colon \Grp \to \VGrp$ equips a group $X$ with the indiscrete
$V$-category structure: $a(x,y)=\top$ for all $x,y\in X$. It is immediate to see
that both structures are compatible with the group operation.

\item \label{complete cocomplete} \emph{\VGrp\ is complete and cocomplete}, as stated in the proof of Proposition \ref{top}. In particular, the initial object is $(\{*\}, \kappa)$ where $\kappa(*,*)=k$, while the terminal object is $(\{*\}, c)$, where $c(*,*)=\top$. Hence $\VGrp$ is a pointed category if and only if, in $V$, $k=\top$.

\item \label{limits} \emph{Limits are preserved by both forgetful functors.} Therefore the product $X\times Y$, of two $V$-groups $(X,a)$ and $(Y,b)$, is the direct product of groups
equipped with the relation $a\wedge b$ given by:
\[ (a\wedge b)((x_1, y_1),(x_2, y_2))=a(x_1,x_2)\wedge b(y_1,y_2).\]
Infinite products are obtained
similarly. The equalizer of a pair $f, g \colon X \to Y$ of
parallel morphisms in $\VGrp$ is the equalizer in $\Grp$
equipped with the $V$-category structure induced by the one
of $X$.

\item \label{colims} \emph{Colimits are preserved by $\U_1\colon\VGrp\to\Grp$ (but not by $\U_2$)}, so they are formed like in \Grp\ and equipped with the suitable $V$-category structure, as outlined next.

    \emph{Coequalizers} are easily described. Given a pair of morphisms $f,g\colon (X,a)\to (Y,b)$, let $q\colon \U_1(Y)\to Q$ be the coequalizer in \Grp\ of $\U_1(f), \U_1(g)$. Defining on $Q$ the structure $c=q\cdot b\cdot q^\circ$, that is $\displaystyle c(z_1,z_2)=\bigvee_{q(y_i)=z_i} b(y_1,y_2)$, thanks to Proposition \ref{final} we know that $q\colon (Y,b)\to (Q,c)$ is a $V$-homomorphism. The universal property is easily checked. This construction shows that \emph{the regular epimorphisms in \VGrp\ are exactly the surjective $V$-homomorphisms $f\colon(X,a)\to(Y,b)$ such that $b=f\cdot a\cdot f^\circ$}.

    \emph{Coproducts} are a bit more difficult. As a group, the coproduct $(Y,b)$ of a family $(X_i,a_i)$ of $V$-groups is the coproduct formed in \Grp\, with the final structure with respect to the forgetful functor into \Grp. There is no easy way of describing this final structure.

\item \label{regular monos} In $\VGrp$ a $V$-homomorphism $f\colon (X,a)\to (Y,b)$ is an \emph{epimorphism} if and only if it is surjective: the preservation of colimits by $\U_1$ and its faithfulness imply that $\U_1$ preserves and reflects epimorphisms; therefore $f$ is an epimorphism in $\VGrp$ if and only if $\U_1(f)$ is an epimorphism in \Grp, that is, $f$ is surjective. \emph{Regular monomorphisms} in $\VGrp$ are the morphisms $f\colon (X,a)\to (Y,b)$ that are injective and with $a(x,x')=b(f(x),f(x'))$, for every $x,x'\in X$. It is easily seen that \emph{(epi, reg mono) is a stable factorization system in \VGrp}: every $f\colon X\to Y$ can be factored as
    \[\xymatrix{(X,a)\ar[rr]^f\ar[rd]_(0.4)e&&(Y,b),\\
    &(f(X),b)\ar[ru]_(0.6)m}\]
    and epimorphisms are pullback stable (just because surjective homomorphisms are pullback stable in \Grp).

\item \label{VGrp regular when V frame} From the construction of coequalizers it follows that a morphism $f\colon (X,a)\to (Y,b)$ is a \emph{regular epimorphism} in $\VGrp$ if, and only if, it is surjective and \emph{final}, that is, $b(y_1,y_2)=\displaystyle\bigvee_{f(x_i)=y_i}a(x_1,x_2)$; and $f$ is a \emph{monomorphism} exactly when it is an injective map. Next we prove that \emph{(reg epi, mono) is a stable factorization system in \VGrp}. Indeed, every $f\colon (X,a)\to (Y,b)$ can be factored as
    \[\xymatrix{(X,a)\ar[rr]^f\ar[rd]_(0.4)e&&(Y,b),\\
    &(f(X),c)\ar[ru]_(0.6)m}\]
    with $c(y_1,y_2)=\displaystyle\bigvee_{f(x_i)=y_i}a(x_1,x_2)$. Pullback-stability of regular epimorphisms is easily checked: if
    \[ \xymatrix{ (X \times_Y Z,a\wedge c) \ar[r]^-{\pi_2} \ar[d]_-{\pi_1} & (Z,c)
\ar[d]^g \\
(X,a) \ar[r]_f & (Y,b) } \]
is a pullback diagram and $f$ is a regular epimorphism, then $\pi_2$ is surjective and
\[\begin{array}{rcl}
c(z_1,z_2)&=& b(g(z_1),g(z_2))\wedge c(z_1,z_2)\\
&=&\displaystyle\left(\bigvee_{f(x_i)=g(z_i)} a(x_1,x_2)\right)\wedge c(z_1,z_2)\\
&=&\displaystyle\bigvee_{f(x_i)=g(z_i)} \left(a(x_1,x_2)\wedge c(z_1,z_2)\right);
\end{array}\]
hence $\pi_2$ is also final.
\end{enumerate}
\end{remark}

\begin{prop}
The category \VGrp\ is a regular category in the sense of Barr \cite{Barr}. Moreover, when $V$ is integral (i.e. $k = \top$) it is normal in the sense of Z. Janelidze \cite{ZJanelidze2010}: every regular epimorphism is a cokernel.
\end{prop}

\begin{proof}
As stated above, $\VGrp$ is a (finitely) complete category with a stable factorization system (reg epi, mono), hence it is regular. To show that it is normal when $V$ is integral, we observe that it is pointed and that, for every regular epimorphism $f\colon (X, a)\to (Y,b)$, $\U_1(f)$ is a regular epimorphism in \Grp, hence it is the cokernel of its kernel in \Grp. Then $f$ is the cokernel of its kernel also in $\VGrp$, indeed, thanks to Remark \ref{functors}.\eqref{colims}, the structure $b$ on $Y$ is the final one: $b = f \cdot a \cdot f^{\circ}$.
\end{proof}

In general, the category $\VGrp$ is not Barr-exact \cite{Barr}: see, for example, Remark $2.6$ in \cite{CMFM19} for an example of an equivalence relation which is not effective, in the case $V = \two$. However, $\VGrp$ satisfies a slightly weaker property: it is \emph{efficiently regular} in the sense of \cite{Bourn Baersums}:

\begin{prop}
$\VGrp$ is efficiently regular: it is regular and, if $R$ is an effective equivalence relation over an object $X$ and $T$ is another equivalence relation over $X$ which is a \emph{regular subobject} $j \colon T \rightarrowtail R$ of $R$ (i.e. $j$ is a regular monomorphism in $\VGrp$), then $T$ is itself effective.
\end{prop}

\begin{proof}
If $\xymatrix{T \ar@<-2pt>[r]_{t_2} \ar@<2pt>[r]^{t_1} & X}$ is an equivalence relation in $\VGrp$ as in the statement above, then $\U_1(T)$ is a kernel pair of a morphism in $\Grp$, since \Grp\ is Barr-exact; hence the following is a pullback in \Grp\:
\[ \xymatrix{ U_1(T) \ar[r]^{U_1(t_1)} \ar[d]_{U_1(t_2)} & U_1(X) \ar[d]^q \\
U_1(X) \ar[r]_q & Y, } \]
where $q \colon U_1(X) \to Y$ is the coequalizer of $U_1(t_1)$ and $U_1(t_2)$ in $\Grp$. Putting on $Y$ the final structure described in Remark \ref{functors}.\eqref{colims}, the square above becomes a commutative diagram in $\VGrp$. It is a pullback in $\VGrp$, because $\langle t_1, t_2 \rangle \colon T \to X \times X$ is a regular monomorphism in $\VGrp$. Indeed, $\xymatrix{R \ar@<-2pt>[r]_{r_2} \ar@<2pt>[r]^{r_1} & X}$ is an effective equivalence relation in $\VGrp$, and so the monomorphism $\langle r_1, r_2 \rangle \colon R \to X \times X$ is a regular monomorphism in $\VGrp$. Moreover, being $j$ a regular monomorphism in $\VGrp$, we obtain from Remark \ref{functors}.\eqref{regular monos} that also the monomorphism $\langle t_1, t_2 \rangle = \langle r_1, r_2 \rangle \circ j$ is regular in $\VGrp$. We conclude that $T$ is effective in $\VGrp$.
\end{proof}

The main reason why this property is interesting is the fact that, in an efficiently regular category, a morphism is effective for descent if and only if it is a regular epimorphism. This follows from \cite[Proposition $1.2$]{Bourn Baersums} and \cite[Theorem $3.7$]{JST}.

\section{Topological properties}

Regularity of topological groups and openness of quotient maps between topological groups play a crucial role in the topological behaviour of the category of topological groups.
Following the approach of \cite{CCT14}, where, inspired by the description of topological spaces and continuous maps as $(\TT,V)$-categories and $(\TT,V)$-functors (for $\TT$ the ultrafilter monad and $V=\two$), several topological properties were explored, we show next that in $\VGrp$ these properties play also a crucial role.

 We start by recalling some notions, studied in \cite{CCT14}, in the context of $V$-categories.

\begin{definition}
A $V$-functor $f\colon(X,a)\to(Y,b)$ is \emph{proper} if $b\cdot f\leq f\cdot a$, while it is \emph{open} if $b^\circ \cdot f\leq f\cdot a^\circ$.
\end{definition}

 Recalling that a map $f\colon(X,a)\to (Y,b)$ is a $V$-functor if $f\cdot a\leq b\cdot f$, or, equivalently, $f\cdot a^\circ \leq b^\circ\cdot f$, one concludes that $f$ is proper exactly when $b\cdot f=f\cdot a$, while $f$ is open when $b^\circ\cdot f=f\cdot a^\circ$. Therefore, $f\colon(X,a)\to(Y,b)$ is proper if, and only if, $f\colon(X,a^\circ)\to(Y,b^\circ)$  is open; using pointwise notation, $f$ is proper when, for $x\in X$, $y\in Y$,
\[b(f(x),y)=\bigvee_{f(x')=y}a(x,x'),\]
and $f$ is open when
\[b(y,f(x))=\bigvee_{f(x')=y} a(x',x).\]

\begin{definition}
 A $V$-category $(X, a)$ is said to be \emph{regular} if $a\cdot a^\circ\leq a$, which is equivalent to $a=a^\circ$, that is, it is symmetric. Indeed, regularity means that, for all $x_1, x_2, x_3 \in X$, one has $a(x_1, x_2) \otimes a(x_1, x_3) \leq a(x_2, x_3)$. Clearly, if $(X, a)$ is symmetric, then this condition is satisfied. Conversely, from regularity one obtains, choosing $x_1=x_3$:
 \[ a(x_3, x_2) = a(x_3, x_2) \otimes k \leq a(x_3, x_2) \otimes a(x_3, x_3) \leq a(x_2, x_3) \quad \text{for all} \ x_2, x_3 \in X. \]
\end{definition}
As a side remark we observe that regularity of the $V$-relation $a$ coincides with other properties studied in diverse algebraic settings.

\begin{lemma}
For a $V$-category $(X,a)$, the following conditions are equivalent:
\begin{tfae}
\item $(X,a)$ is regular,
\item $(X,a)$ is a symmetric $V$-category,
\item $a$ is a \emph{positive} $V$-relation (see \cite{Sel07}), i.e. $a=b^\circ\cdot b$ for some $V$-relation $b$,
\item $a$ is \emph{difunctional} (see \cite{Ri48}), i.e. $a\cdot a^\circ\cdot a\leq a$. \hfill $\Box$
\end{tfae}
\end{lemma}

\begin{corollary}
A $V$-group is regular if, and only if, it is symmetric; hence, when $\otimes=\wedge$, a $V$-group is regular if, and only if, it is an internal group in \VCat. \hfill $\Box$
\end{corollary}

\begin{prop}
A $V$-homomorphism $f\colon(X,a)\to(Y,b)$ between $V$-groups is open if, and only if, it is proper.
\end{prop}

\begin{proof}
For every $x\in X$, $y\in Y$, $b(f(x),y)=b(-y,-f(x))=b(-y,f(-x))$, and the result follows.
\end{proof}

\begin{remark}
For inclusions, the notion of proper $V$-functor is the right substitute for closed subobject. In fact, as proper map in topology means \emph{stably closed map}, and closed embeddings are pullback stable, a topological embedding is closed if, and only if, it is proper. Having this in mind, the following statement corresponds to the well-known fact that every open subgroup of a topological group is closed.
\end{remark}

\begin{corollary}
If $S$ is a subgroup of $(X,a,+)$, then $S$ is open in $X$ if and only if it is proper. \hfill $\Box$
\end{corollary}

\begin{theorem}
Every regular epimorphism in \VGrp\ is both open and proper.
\end{theorem}

\begin{proof}
As we checked in Remark \ref{functors}, a regular epimorphism $f\colon(X,a)\to(Y,b)$ in \VGrp\ is a surjective map such that $b=f\cdot a\cdot f^\circ$, that is, for every $y_1,y_2\in Y$, $\displaystyle b(y_1,y_2)=\bigvee_{f(x_i)=y_i} a(x_1,x_2)$. Hence,
\[b(y,f(x))=\bigvee_{f(x')=y,\,f(x'')=f(x)} a(x',x'')=\bigvee_{f(x')=y,\,f(x'')=f(x)} a(x'-x''+x,x)=\bigvee_{f(z)=y} a(z,x).\]
\end{proof}

\section{Algebraic properties} \label{Section alg prop}

The aim of this section is to study the algebraic properties of $V$-groups. \emph{From now on, we will always assume that, in $V$, the equality $k = \top$ holds.} This assumption is not particularly restrictive: all the examples we mentioned in the previous sections satisfy it. But, as we observed in Remark \ref{functors}.\eqref{complete cocomplete}, it implies that the category $\VGrp$ is pointed. Note that, if $k = \top$, then the condition for $(X, a)$ to be a $V$-graph becomes $a(x, x) = k = \top$ for all $x \in X$.

Moreover, due to Proposition \ref{unital} below, we will concentrate mainly in the case of $V$ being a frame, i.e. $\otimes = \wedge$ in $V$: under this assumption, we will explore the main categorical-algebraic notions, showing that, in general, the whole category $\VGrp$ satisfies only relatively weak properties. However, if we restrict our attention to symmetric $V$-groups, we will see that the algebraic behaviour is very similar to the one of the category $\Grp$ of groups. We start by observing that the property of being unital holds in $\VGrp$ if and only if $V$ is a frame.

We recall that a pair of morphisms with the same codomain in a finitely complete category is \emph{jointly strongly epimorphic} if, whenever both morphisms factor through a monomorphism $m$, $m$ is an isomorphism. A pointed, finitely complete category is \emph{unital} \cite{Bourn unital} if, for any two objects $X$ and $Y$, the canonical morphisms
\[\xymatrix{X\ar[r]^-{\langle 1, 0 \rangle}& X \times Y& Y,\ar[l]_-{\langle 0, 1 \rangle}}\]
 induced by the universal property of the product, are jointly strongly epimorphic. We remark that, for $(X,a)$ and $(Y,b)$ $V$-groups, the $V$-category structure on $(X,a)\times(Y,b)$ is $a\wedge b$.

\begin{prop}\label{unital}
$\VGrp$ is a unital category if and only if $V$ is a frame.
\end{prop}

\begin{proof}
Suppose first that $V$ is a frame. Given $X, Y, Z \in \VGrp$, consider the following commutative
diagram
\[ \xymatrix{ X \ar[r]^-{\langle 1, 0 \rangle} \ar[dr]_f & X
\times Y & Y, \ar[l]_-{\langle 0, 1 \rangle} \ar[dl]^g \\
& Z \ar[u]_m & } \] where $m$ is a monomorphism. Since
$\Grp$ is a unital category, $m$ is an isomorphism of groups; it only remains to
show that its inverse $t$ is a $V$-functor. The morphism $t$ is defined by $t(x, y) = f(x) + g(y)$. In other terms, $t = + \circ (f \times g)$. Hence $t$ is a $V$-functor, being the composite of two $V$-functors.

Conversely, suppose that $V$ is not a frame. Hence there exist $u, v \in V$ such that $u \otimes v < u \wedge v$ (observe that, under the assumption that $k = \top$, it is always true that $u \otimes v \leq u \wedge v$ for all $u, v \in V$). Given any pair $(X, a)$ and $(Y, b)$ of $V$-groups, the identity map
\[\xymatrix{(X,a)\otimes(Y,b)\ar[r]^-{1_{X \times Y}}&(X,a) \times (Y,b)}\]
is a monomorphism in $\VGrp$. Moreover, we always have the following commutative diagram in $\VGrp$:
\[ \xymatrix{ (X, a) \ar[r]^-{\langle 1, 0 \rangle} \ar[dr]_{1 \otimes 0} & (X
\times Y, a \wedge b) & (Y, b), \ar[l]_-{\langle 0, 1 \rangle} \ar[dl]^{0 \otimes 1} \\
& (X \times Y, a \otimes b) \ar[u]_{1_{X \times Y}} & } \]
where $(1 \otimes 0)(x)=(x, 0)$ and $(0 \otimes 1)(y)=(0,y)$. We want to show that $1_{X \times Y}$ is not always an isomorphism. We define, on the additive group $\mathbb{Z}_2 = \mathbb{Z} / 2\mathbb{Z}$, two $V$-group structures $a$ and $b$ as follows:
\[ a(0, 1) = a(1, 0) = u, \quad a(0,0) = a(1,1) = k; \qquad b(0, 1) = b(1, 0) = v, \quad b(0,0) = b(1,1) = k. \]
Then
\[ (a \otimes b)((0,0), (1,1)) = a(0, 1) \otimes b(0,1) = u \otimes v < u \wedge v = a(0, 1) \wedge b(0,1) = (a \wedge b)((0,0), (1,1)). \]
\end{proof}

One of the key categorical-algebraic notions is the one of \emph{protomodular} category \cite{Bourn protomod}: for pointed, finitely complete categories, it is equivalent to the validity of the Split Short Five Lemma, and it has several important consequences, mainly related to the homological properties of the category. Unfortunately, if $V$ is non-degenerate (i.e. $\bot \neq \top$), $\VGrp$ is not a protomodular category. Indeed, as shown in \cite{CMFM19}, \OrdGrp\ is not protomodular. Moreover, when $V$ is integral and non-degenerate, using the inclusion $\iota \colon \two \to V$ described in Section \ref{Section V-categories} it is not difficul to see that \OrdGrp\ can be identified with a full subcategory of $\VGrp$ closed under finite limits. Therefore $\VGrp$ is not a protomodular category, since it has a full subcategory, \OrdGrp, which is closed under finite limits and not protomodular. However, when $V$ is a frame, $\VGrp$ is good enough to identify, inside of it, an important full subcategory, which turns out to be the one of symmetric $V$-groups, which is protomodular (and, actually, it satisfies even stronger algebraic properties, as we will see in Section \ref{Section sym Vgroups}). In order to see this, we follow the \emph{objectwise} approach introduced in \cite{MRVdL objects}: the idea is to identify a class of objects, in a category with weak algebraic properties, such that the main categorical-algebraic properties hold locally for constructions involving these ``good'' objects. In order to define formally such objects, we need the notion of stably strong point:

\begin{definition}
A point (i.e. a split epimorphism with a fixed section) $\xymatrix{ A \ar@<-2pt>[r]_f & Y \ar@<-2pt>[l]_s }$ with kernel $n \colon X \to A$ in a pointed finitely complete category is \emph{strong} if $n$ and $s$ are jointly strongly epimorphic. It is \emph{stably strong} if every pullback of it along any morphism $g \colon Z \to Y$ is strong.
\end{definition}

\begin{definition}[\cite{MRVdL objects}]
An object $Y$ of a pointed, finitely complete category $\C$ is
\begin{itemize}
\item[(1)] a \emph{strongly unital object} if the point $\xymatrix{ Y \times Y \ar@<-2pt>[r]_-{\pi_2} & Y \ar@<-2pt>[l]_-{\langle 1, 1 \rangle} }$ is stably strong;

\item[(2)] a \emph{Mal'tsev object} if, for every pullback of split
epimorphisms over $Y$ as in the following diagram
\[
\vcenter{\xymatrix@!0@=5em{ A\times_{Y}Z \ar@<-.5ex>[d]_{\pi_A}
\ar@<-.5ex>[r]_(.7){\pi_Z} & Z \ar@<-.5ex>[d]_g
\ar@<-.5ex>[l]_-{\langle sg,1_Z \rangle} \\
A \ar@<-.5ex>[u]_(.4){\langle 1_A,tf \rangle} \ar@<-.5ex>[r]_f &
Y, \ar@<-.5ex>[l]_s \ar@<-.5ex>[u]_t }} \] the morphisms $\langle
1_{A}, tf \rangle$ and $\langle sg, 1_Z \rangle$ are jointly
strongly epimorphic;

\item[(3)] a \emph{protomodular object} if every point over $Y$ is stably strong.
\end{itemize}
\end{definition}

\begin{prop}
If $(Y,b,+)$ is a strongly unital $V$-group, then, for every $y\in Y$,
\[b(0,y)=b(y,0)\otimes b(0,y).\]
\end{prop}

\begin{proof}
Assume there is $x\in Y$ such that $b(x,0)\otimes b(0,x)\neq b(0,x)$. Consider the subgroup $X$ of $Y$ generated by $x$, equipped with the $V$-categorical structure induced by $b$ (and that we will denote also by $b$). We show next that, in the following diagram
\[ \xymatrix{ Y \ar@{=}[d] \ar[r]^-{\langle 1, 0 \rangle} & Y \times X \ar[d]^{1 \times j} \ar@<-2pt>[r]_-{\pi_2} & X \ar@<-2pt>[l]_-{\langle j, 1 \rangle} \ar[d]^j \\
Y \ar[r]^-{\langle 1, 0 \rangle} & Y \times Y \ar@<-2pt>[r]_-{\pi_2} & Y, \ar@<-2pt>[l]_-{\langle 1, 1 \rangle} } \]
with $X\times Y$ equipped with the product structure
\[d((0,0),(y,y'))=b(0,y)\wedge b(0,y'),\]
$\xymatrix{(Y,b)\ar[r]^-{\langle 1,0\rangle}&(Y\times X,d)&(X,b)\ar[l]_-{\langle j,1\rangle}}$ are not jointly strongly epimorphic, that is, $d$ is not the final structure $c$ for $(\langle 1,0\rangle,\langle j,1\rangle)$.

To define $c$ we first note that, since $(y,z)=\langle 1,0\rangle(y-z)+\langle j,1\rangle(z)$, necessarily
\[c((0,0),(y,z))\geq b(0,y-z)\otimes b(0,z)=b(z,y)\otimes b(0,z).\]
Let us define $c_0\colon (Y\times X)\times(Y\times X)\to V$ by:
\[c_0((0,0),(y,z))=b(z,y)\otimes b(0,z),\]
and extend it by shifting. Then $(Y\times X,c_0)$ is a $V$-graph and, since $c_0$ may be neither transitive nor compatible with the addition on $Y\times X$, we define then $c$ by
\[c((0,0),(y,z))=\bigvee c_0((0,0),(y_1,z_1))\otimes \dots\otimes c_0((0,0),(y_n,z_n)),\]
where $y_1+\dots+y_n=y$ and $z_1+\dots+z_n=z$, and extend it by shifting.

To show that $(Y\times X,c,+)$ is a $V$-group, thanks to Proposition \ref{p:categ}, it is enough to show that $+\colon (Y\times X)\times (Y\times X)\to Y\times X$ is a $V$-functor. Let $y,y'\in Y$, $z,z'\in X$; then
\[c((0,0),(y,z))\otimes c((0,0),(y',z'))=\]
\[=\bigvee c_0((0,0),(y_1,z_1))\otimes\dots\otimes c_0((0,0),(y_n,z_n))\otimes c_0((0,0),(y'_1,z'_1))\otimes\dots\otimes c_0((0,0),(y'_m,z'_m)),\]
where $y_1+\dots y_n=y$, $z_1+\dots+z_n=z$, $y'_1+\dots+y'_m=y'$, $z'_1+\dots+z'_m=z'$, and this is clearly less or equal to
\[\bigvee c_0((0,0),(y_1^*,z_1^*))\otimes\dots\otimes c_0((0,0),(y_l^*,z_l^*)),\]
with $y_1^*+\dots+y_l^*=y+y'$, $z_1^*+\dots z_l^*=z+z'$, that is, $c((0,0),(y+y',z+z'))$.

Finally we are going to show that \[c((0,0),(0,x))=c_0((0,0),(0,x))=b(x,0)\otimes b(0,x),\]
which shows that $c\neq d$ since $d((0,0),(0,x))=b(0,x)$, which is different from $b(x,0)\otimes b(0,x)$ by hypothesis. On one hand, we always have $c_0 \leq c$. On the other hand, given $y_1,\dots,y_n\in Y$, $z_1,\dots,z_n\in X$ such that $y_1+\dots+y_n=0$ and $z_1+\dots+z_n=x$,
\begin{align*}c_0((0,0),(y_1,z_1))&\otimes\dots\otimes c_0((0,0),(y_n,z_n))=\\
&=b(z_1,y_1)\otimes b(0,z_1)\otimes \dots\otimes b(z_n,y_n)\otimes b(0,z_n)\\
&=b(z_1,y_1)\otimes\dots\otimes b(z_n,y_n)\otimes b(0,z_1)\otimes\dots\otimes b(0,z_n)\\
&\leq b(z_1+\dots +z_n,y_1+\dots +y_n)\otimes b(0,z_1+\dots+z_n)=b(x,0)\otimes b(0,x),
\end{align*}
which shows that
\[ c((0,0),(0,x))=\bigvee c_0((0,0),(y_1,z_1))\otimes \dots\otimes c_0((0,0),(y_n,z_n)),\]
where $y_1+\dots+y_n=0,\; z_1+\dots+z_n=x$,
is less than or equal to $b(x,0)\otimes b(0,x)$.
\end{proof}

\begin{theorem}
When $\otimes=\wedge$ in $V$, for a $V$-group $(Y,b,+)$, the following conditions are equivalent:
\begin{tfae}
\item $(Y,b,+)$ is a protomodular object in $\VGrp$;
\item $(Y,b,+)$ is a Mal'tsev object in $\VGrp$;
\item $(Y,b,+)$ is a strongly unital object in $\VGrp$;
\item $(Y,b)$ is a symmetric $V$-category;
\item $(Y,b,+)$ is an internal group in $\VCat$.
\end{tfae}
\end{theorem}

\begin{proof}
The implications (i)$\implies$(ii)$\implies$(iii) follow from Propositions $7.2$ and $6.3$ in \cite{MRVdL objects}, because $\VGrp$ is a regular category, as observed in Remark \ref{functors}.\eqref{VGrp regular when V frame}.

(iii)$\implies$(iv)$\iff$(v): follows from the previous proposition since, whenever $\otimes=\wedge$, from
\[ b(0,y)=b(y,0)\wedge b(0,y) = b(0, -y) \wedge b(-y, 0) = b(-y, 0) \wedge b(0, -y) = b(0, -y) \]
we conclude that $b(0,y)=b(0,-y)=b(y,0)$, and therefore $(Y,b)$ is a symmetric $V$-category, or, equivalently, the inversion is a $V$-functor, as we observed in Remark \ref{inversion V-functor iff symmetric}.

It remains to show that (v)$\implies$(i): given an internal group $(Y,b,+)$, consider the following diagram in $\VGrp$
\[\xymatrix{ N \ar[d] \ar[r]^-{n} & (Z \times_Y X,d)
\ar@<-2pt>[d]_{f'} \ar[r]^-{h'} & (X,a) \ar@<-2pt>[d]_f \\
0 \ar[r] & (Z,c) \ar@<-2pt>[u]_{\langle 1_Z, sh \rangle} \ar[r]_h & (Y,b),
\ar@<-2pt>[u]_s }\]
where both downwards squares are pullbacks. The pullback structure $d$ on $Z\times_YX$ is given by
\[d((0,0),(z,x))=c(0,z)\wedge a(0,x),\]
for every $(z,x)\in Z\times_YX$. To show that $d$ coincides with the final structure $d'$ for $n$ and $\langle 1_Z,sh\rangle$, we note that, for every $(z,x)\in Z\times_YX$, \[(z,x)=(0,x-sf(x))+(z,sf(x))=(0,x-sf(x))+(z,sh(z)).\]
From $V$-functoriality of $\langle 1_Z,sh\rangle$ one gets immediately that
\[d'((0,0),(z,sh(z))\geq c(0,z).\]
Now, using $V$-functoriality of $s$ and $f$, and symmetry of $(Y,b)$,
\begin{align*}
d'((0,0),(0,x-sf(x))\geq& a(0,x-sf(x))\geq a(0,x)\wedge a(0,s(-f(x)))\\
\geq& a(0,x)\wedge b(0,-f(x))=a(0,x)\wedge b(0,f(x))=a(0,x).
\end{align*}
Therefore $d' \geq d$. In other terms, the identity morphism $1_{Z \times_Y X} \colon (Z \times_Y X, d) \to (Z \times_Y X, d')$ is a $V$-functor. But $1_{Z \times_Y X}$ is clearly a monomorphism, and $n$ and $\langle 1_Z, sh \rangle$ factor through it. Being $d'$ the final structure for these two morphisms, we get that $1_{Z \times_Y X} \colon (Z \times_Y X, d) \to (Z \times_Y X, d')$ is an isomorphism, i.e. $d' = d$.
\end{proof}

As a reflective subcategory of $\VGrp$ (see Proposition \ref{p:sym}), $\VGrp_\sym$ is closed under limits in $\VGrp$. Hence $\VGrp_\sym$ is a protomodular category, thanks to \cite[Corollary $7.4$]{MRVdL objects}. Indeed, by the previous theorem, $\VGrp_\sym$ is the full subcategory of protomodular objects of $\VGrp$. In particular we get the following examples which give rise to protomodular categories:
\begin{enumerate}
\item When $V=2$, $\OrdGrp_\sym$ is the category of groups equipped with a congruence, or in other terms the category whose objects are pairs $(G, N)$ where $G$ is a group and $N$ is a normal subgroup of $G$, and whose arrows are the group homomorphisms that restrict to the normal subgroups.

\item When $V=P_\max$, $\UMetGrp_\sym$ is the category of symmetric ultrametric groups, i.e. ultrametric groups in which the distance is symmetric, and non-expansive homomorphisms. Moreover, if we consider the full subcategory $\UMetGrp_{\sym,0}$ of (classical) ultrametric groups, i.e. of those ultrametric groups that are symmetric, \emph{separated} (if $d(x,y)=0$, then $x=y$) and \emph{finitary} (for all $x, y$ $d(x, y) < \infty$), it is easy to observe that this subcategory is closed in $\VGrp_\sym$ under finite limits, hence it is itself protomodular (since the notion of protomodularity can be expressed only by means of finite limits).

\item When $V=\Delta_\wedge$, $\ProbMetGrp_\sym$ is the category of symmetric probabilistic ultrametric groups.
\end{enumerate}
In Section \ref{Section sym Vgroups} we will investigate more in detail the algebraic properties of $\VGrp_\sym$.

\section{Split extensions} \label{Section split ext}

In this section we investigate the split extensions in $\VGrp$. \emph{We will always assume that $k = \top$ in $V$, so that $\VGrp$ is pointed, but we do not require that $V$ is a frame.} Let \linebreak
$\xymatrix{(X,a)\ar[r]^n&(Z,c)\ar@<-.5ex>[r]_f&(Y,b)\ar@<-.5ex>[l]_s}$ be a split extension in $\VGrp$. Then $\xymatrix{X\ar[r]^n&Z\ar@<-.5ex>[r]_f&Y\ar@<-.5ex>[l]_s}$ is a split extension in \Grp, hence $Z$, as a group, is isomorphic to the semidirect product $X\rtimes_\varphi Y$ with respect to the action $\varphi\colon Y\to \Aut(X)$ of $Y$ on $X$ defined by $\varphi(y)(x)=\varphi_y(x)=s(y)+n(x)-s(y)$. Therefore we can restrict our study to split extensions of the type
\begin{equation}\label{split}
\xymatrix{(X,a)\ar[r]^-{\langle 1,0\rangle}&(X\rtimes_\varphi Y,c)\ar@<-.5ex>[r]_-{\pi_2}&(Y,b)\ar@<-.5ex>[l]_-{\langle 0,1\rangle}}
\end{equation}
We recall that, for $(x,y),(x',y')\in X\times Y$, $(x,y)+_\varphi(x',y')=(x+\varphi_y(x'),y+y')$. In general we will omit the index $\varphi$ in the sum $+$.

First we present a necessary condition for a $V$-category structure $c$ on $X\rtimes_\varphi Y$ to make \eqref{split} a split extension in $\VGrp$.

\begin{lemma}
If \eqref{split} is a split extension in \VGrp\, then, for every $y\in Y$, $\varphi_y\colon (X,a)\to(X,a)$ is a $V$-functor.
\end{lemma}

\begin{proof}
Invariance of $a$ by shifting gives that, for every $x,x'\in X$,
\begin{align*}
a(x,x')=&c(n(x),n(x'))=c(s(y)+n(x)-s(y),s(y)+n(x')-s(y))=c(\varphi_y(x),\varphi_y(x'))\\=&a(\varphi_y(x),\varphi_y(x')).
\end{align*}
\end{proof}

As for preordered groups, there are two possible extremal structures to be considered, a minimal one given by the tensor $\otimes$ in the product, and a maximal one, a \emph{generalized lexicographic structure} we introduce below.

\begin{theorem}\label{t:tensor} %[Proposition 5.2]
Let $X$ and $Y$ be groups, $\varphi \colon Y \to \Aut(X)$ a group action, and let $(X\rtimes Y,+)$ be the semidirect product defined in $\Grp$ by $\varphi$. The following assertions are equivalent:
\begin{tfae}
\item $\xymatrix{(X,a)\ar[r]^-{\langle 1,0\rangle}&(X\rtimes Y,a\otimes b)\ar@<-.5ex>[r]_-{\pi_2}&(Y,b)\ar@<-.5ex>[l]_-{\langle 0,1\rangle}}$ is a split extension in $\VGrp$;
\item the map $\overline{\varphi} \colon (X\otimes Y,a\otimes b)\to (X\otimes Y,a\otimes b)$, with $(x,y)\mapsto (\varphi_y(x),y)$, is a $V$-functor.
\end{tfae}
\end{theorem}

\begin{proof}
(i)$\implies$(ii): $(X\rtimes Y,a\otimes b,+)$ is a $V$-group if, and only if, for all $x_1,x_2,x_1',x_2'\in X$, $y_1,y_2,y_1',y_2'\in Y$,
\[(a\otimes b)((x_1,y_1),(x_2,y_2))\otimes (a\otimes b)((x_1',y_1'),(x_2',y_2'))\leq (a\otimes b)((x_1,y_1)+(x_1',y_1'),(x_2,y_2)+(x_2',y_2'));\]
that is,
\[a(x_1,x_2)\otimes b(y_1,y_2)\otimes a(x_1',x_2')\otimes b(y_1',y_2')\leq a(x_1+\varphi_{y_1}(x_1'),x_2+\varphi_{y_2}(x_2'))\otimes b(y_1+y_1', y_2+y_2').\]
When $x_1=x_2=0$ and $y_1'=y_2'=0$, we obtain
\[a(x_1',x_2')\otimes b(y_1,y_2)\leq a(\varphi_{y_1}(x_1'),\varphi_{y_2}(x_2'))\otimes b(y_1,y_2),\]
and this means exactly that $\overline{\varphi}$ is a $V$-functor.

(ii)$\implies$(i): To verify that $(X\rtimes Y,a\otimes b,+)$ is a $V$-group, let $x_1,x_2,x_1',x_2'\in X$ and $y_1,y_2,y_1',y_2'\in Y$; then
\begin{align*}
(a\otimes b)((x_1,y_1),(x_2,y_2))&\otimes (a\otimes b)((x_1',y_1'),(x_2',y_2'))\\
=&a(x_1,x_2)\otimes b(y_1,y_2)\otimes a(x_1',x_2')\otimes b(y_1',y_2')\\
=&a(x_1,x_2)\otimes a(x_1',x_2')\otimes b(y_1,y_2)\otimes b(y_1',y_2')\\
\leq& a(x_1,x_2)\otimes a(\varphi_{y_1}(x_1'),\varphi_{y_2}(x_2'))\otimes b(y_1,y_2)\otimes b(y_1',y_2')&\mbox{($\overline{\varphi}$ is a $V$-functor)}\\
\leq& a(x_1+\varphi_{y_1}(x_1'),x_2+\varphi_{y_2}(x_2'))\otimes b(y_1+y_1',y_2+y_2')\\
=&(a\otimes b)((x_1,y_1)+(x_1',y_1'),(x_2,y_2)+(x_2',y_2')).
\end{align*}
It remains to show that the homomorphisms of the split extension are $V$-functors, and that $\langle 1_X,0\rangle$ is a kernel. The monomorphisms $\langle 1,0\rangle$ and $\langle 0,1\rangle$ are always $V$-functors, as well as $\pi_2$, since it means that $a(x,x')\otimes b(y,y')\leq b(y,y')$, and this is true because we are assuming that $k = \top$; moreover, for every $x,x'\in X$, $a(x,x')=(a\otimes b)((x,0),(x',0))$, and so $(X,a)$ has the initial structure for the map $\langle 1, 0\rangle$.
\end{proof}

\begin{remark}
\begin{enumerate}
\item Theorem \ref{t:tensor} applied to the case $V=\two$ gives Proposition 5.2 of \cite{CMFM19}: \emph{when $V=\two$, $(X\rtimes Y,a\otimes b,+)$ is a $V$-group if and only if
\[(\forall \, y\geq 0)\;\;(\forall \, x\in X)\;\;\varphi_y(x)\geq x.\]}
Indeed, $V$-functoriality of $\overline{\varphi}$ gives:
\[b(0,y)=(a\otimes b)((x,0),(x,y))\leq (a\otimes b)((\varphi_0(x),0),(\varphi_y(x),y))=a(x,\varphi_y(x))\otimes b(0,y).\]
Then, when $V=\two$ and $y\geq 0$, i.e. $b(0,y)=\top$, $a(x,\varphi_y(x))\wedge\top\geq\top$ means exactly that $\varphi_y(x)\geq x$.

\item The same theorem, applied to the case $V = P_+$, says that $(X\rtimes Y,a\otimes b,+)$ is a $V$-group if and only if
\[ \forall \, x_1, x_2 \in X,\;\forall \, y_1, y_2 \in Y \mbox{ with }b(y_1,y_2)\neq\infty, \quad a(x_1, x_2) \geq a(\varphi_{y_1}(x_1), \varphi_{y_2}(x_2)). \]
Indeed, $\overline{\varphi}$ is a $V$-functor (i.e. a non-expansive map) if and only if, for all $x_1, x_2 \in X, y_1, y_2 \in Y$
\[ (a\otimes b)((x_1, y_1),(x_2, y_2)) \geq (a\otimes b)((\varphi_{y_1}(x_1), y_1),(\varphi_{y_2}(x_2), y_2)), \] which is the same as to say that
\[ a(x_1, x_2) + b(y_1, y_2) \geq a(\varphi_{y_1}(x_1), \varphi_{y_2}(x_2)) + b(y_1, y_2). \]
\end{enumerate}
\end{remark}

Next we analyse how to interpret the \emph{lexicographic structure} in $\VGrp$. For $V$-categories $(X,a)$, $(Y,b)$, consider
$\lex\colon (X\times Y)\otimes(X\times Y)\to V$ defined by
\[\lex((x,y),(x',y'))=\left\{\begin{array}{ll}
a(x,x')&\mbox{ if }y=y'\\
b(y,y')&\mbox{ if }y\neq y'.\end{array}\right.\]
In general $(X\times Y,\lex)$ is a $V$-graph but not necessarily a $V$-category, as shown in the proof of the theorem below.

\begin{theorem}\label{lexic}
Given $V$-groups $(X,a), (Y,b)$, and a group action $\varphi\colon Y\to \Aut(X)$ with $\varphi_y\colon(X,a)\to (X,a)$ a $V$-functor for every $y\in Y$, the following conditions are equivalent:
\begin{tfae}
\item $\xymatrix{(X,a)\ar[r]^-{\langle 1,0\rangle}&(X\rtimes Y,\lex)\ar@<-.5ex>[r]_-{\pi_2}&(Y,b)\ar@<-.5ex>[l]_-{\langle 0,1\rangle}}$ is a split extension in $\VGrp$;
\item for all $x\in X$ and $y\in Y\setminus\{0\}$, $b(y,0)\otimes b(0,y)\leq a(x,0)$.
\end{tfae}
\end{theorem}

\begin{proof}
(ii)$\implies$(i): Thanks to Proposition \ref{p:categ}, to show that $(X\rtimes Y,\lex)$ is a $V$-group it is enough to show that $+$ is a $V$-functor; that is, for each $x_1,x_2,x_1',x_2'\in X$, $y_1,y_2,y_1',y_2'\in Y$,
\[\lex((x_1,y_1),(x_2,y_2))\otimes\lex((x_1',y_1'),(x_2',y_2'))\leq \lex((x_1+\varphi_{y_1}(x_1'),y_1+y_1'),(x_2+\varphi_{y_2}(x_2'),y_2+y_2')).\]
For that we consider the possible cases:
\begin{itemize}
\item $y_1+y_1'=y_2+y_2'$, $y_1=y_2$, $y_1'=y_2'$:
\begin{align*}
a(x_1,x_2)\otimes a(x_1',x_2')&\leq a(x_1,x_2)\otimes a(\varphi_{y_1}(x_1'),\varphi_{y_1}(x_2'))&\mbox{($\varphi_{y_1}$ is a $V$-functor)}\\
&\leq a(x_1+\varphi_{y_1}(x_1'),x_2+\varphi_{y_1}(x_2'))&\mbox{($(X,a,+)$ is a $V$-group)}.
\end{align*}
\item $y_1+y_1'=y_2+y_2'$ and $y_1\neq y_2$ (and so $-y_2+y_1=y_2'- y_1'\neq 0$): Using (ii) we conclude that
\[b(y_1,y_2)\otimes b(y_1',y_2')=b(-y_2+y_1,0)\otimes b(0,y_2'-y_1')\leq a(x_1+\varphi_{y_1}(x_1'),x_2+\varphi_{y_2}(x_2')).\]
\item $y_1+y_1'\neq y_2+y_2'$ and $y_1=y_2$:
\[a(x_1,x_2)\otimes b(y_1',y_2')\leq b(y_1',y_2')= b(y_1+y_1',y_2+y_2').\]
\item $y_1+y_1'\neq y_2+y_2'$ and $y_1\neq y_2$, $y_1'\neq y_2'$: immediate.
\end{itemize}
From the definition of $\lex$ it is clear that all the maps of (i) are $V$-functors, and, moreover, $\langle 1,0\rangle$ is the kernel of $\pi_2$.

(i)$\implies$(ii): take $(x,y)+(0,-y)=(x+\varphi_y(0),0)=(x,0)$, for $y\neq 0$: \[b(y,0)\otimes b(-y,0)=\lex((x,y),(0,0))\otimes \lex((0,-y),(0,0))\leq \lex((x,0),(0,0))=a(x,0).\]
\end{proof}

\begin{remark}
\begin{enumerate}
\item Theorem \ref{lexic} applied to the case $V=\two$ gives Proposition 5.4 of \cite{CMFM19}. Indeed, for $V=\two$ condition (ii) is trivially satisfied when the preorder is antisymmetric:
\[(\forall \, y\neq 0)\;\;b(y,0)\wedge b(0,y)=\bot,\]
and, in the presence of a non-positive element $x$ of $X$, so that $a(0,-x)=\bot$, (ii) implies antisymmetry of $b$.

This extends to optimistic quantales. Indeed, if $V$ is optimistic and $X$ is such that $\displaystyle{\wedge_{x\in X} \, a(x,0)=\bot}$, then condition (ii) is valid for $(Y,b)$ if, and only if, for every $y\in Y\setminus\{0\}$, $b(y,0)=\bot$ or $b(0,y)=\bot$. That is, for $V$-groups $(X,a)$, $(Y,b)$, $\lex$ makes the semidirect product $X\rtimes Y$ a $V$-group if, and only if, considering the reflections of $(X,a)$ and $(Y,b)$ in \OrdGrp, the corresponding lexicographic preorder makes $X\rtimes Y$ a preordered group.

In particular, when $V=P_+$ or $V=P_\max$, if $(X,a)$ is a non-bounded (ultra)metric space, then the only $(Y,b)$ which admit the lexicographic order on $X\rtimes Y$ are those with either $b(y,0)=\infty$ or $b(0,y)=\infty$, for any $y\neq 0$.

\item As expected, for symmetric $V$-groups $(Y,b)$ the lexicographic order rarely makes $X\rtimes Y$ a $V$-group: if $(X,a)$ is such that $\displaystyle{\wedge_{x\in X} \, a(x,0)=\bot}$, then, for every $y\neq 0$, $b(y,0)=\bot$, since, for any quantale $V$ and $u\in V$, $u\otimes u=\bot$ implies $u=\bot$.

%\item If $V = P_+$ (i.e. in the case of generalized metric groups), condition (ii) of Theorem \ref{lexic} is much more restrictive. Indeed it becomes:
%\[ \forall \, x \in X, y \in Y\setminus\{0\}, \quad d_Y(y, 0) + d_Y(0, y) \geq d_X(x, 0). \]
%For example, if $Y$ is symmetric and $X$ is such that the supremum of $\{ d_X(x,0) \, | \, x \in X \}$ is infinite, then $d_Y(y_1, y_2) = \infty$ for all $y_1 \neq y_2$. The situation is similar when $V = P_\max$, namely for generalized ultrametric groups.
\end{enumerate}
\end{remark}

Finally we establish the result announced before Theorem \ref{t:tensor}.
\begin{prop}
If $\xymatrix{(X,a)\ar[r]^-{\langle 1,0\rangle}&(X\rtimes Y,c)\ar@<-.5ex>[r]_-{\pi_2}&(Y,b)\ar@<-.5ex>[l]_-{\langle 0,1\rangle}}$ is a split extension in \VGrp, then
\[a\otimes b\leq c\leq \lex.\]
\end{prop}

\begin{proof}
From the equality $(x,0)+(0,y)=(x,y)$, for every $x\in X$ and $y\in Y$, and the fact that $\langle 1_X,0\rangle$ and $\langle 0,1_Y\rangle$ are $V$-functors, it follows that
\[a(0,x)\otimes b(0,y) \leq c((0,0),(x,0))\otimes c((0,0),(0,y))\leq c((0,0),(x,y)).\]
Moreover, $V$-functoriality of $\pi_2$ gives that, for every $y,y'\in Y$, $c((x,y),(x',y'))\leq b(y,y')$. When $y=y'$, $c((0,y),(x,y))=c((0,0),(x,0))=a(0,x)$, because $\langle 1_X,0\rangle$ is a kernel.
\end{proof}

\begin{prop} \label{semidir sym V-groups is product}
When $\otimes=\wedge$ and $(Y,b)$ is a symmetric $V$-category, the only possible compatible structure on $X\rtimes Y$ is $a\otimes b=a\wedge b$.
\end{prop}
\begin{proof}
Let $x\in X$ and $y\in Y$. By $V$-functoriality of $\pi_2$, $c((0,0),(x,y))\leq b(0,y)$. Moreover, since $(x,0)=(x,y)+(0,-y)$, one gets \[c((0,0),(x,y))=c((0,0),(x,y))\wedge b(0,y)=c((0,0),(x,y))\wedge b(0,-y)= \]
\[ = c((0,0),(x,y)) \wedge c((0,0),(0, -y)) \leq c((0,0),(x,0)) =  a(0,x).\]
\end{proof}

\section{Symmetric $V$-groups} \label{Section sym Vgroups}

At the end of Section \ref{Section alg prop} we observed that the category $\VGrp_\sym$ of symmetric $V$-groups is protomodular, when $\otimes = \wedge$. One of the several equivalent ways of formulating this property is in terms of the so-called \emph{fibration of points}: given a finitely complete category $\C$, we denote by $\Pt(\C)$ the category of \emph{points} in \C, i.e. of split epimorphisms with a fixed section. The functor
\[ \cod \colon \Pt(\C) \to \C \]
which associates with every point its codomain is a fibration, called the \emph{fibration of points}. Several categorical-algebraic properties of a category \C\ can be expressed in terms of the change-of-base functors of this fibration. In particular, \C\ is protomodular if and only if, for every morphism $f \colon E \to B$ in \C, the change-of-base $f^* \colon \Pt_B(\C) \to \Pt_E(\C)$ is conservative \cite{Bourn protomod}. For pointed categories, this is equivalent to the validity of the Split Short Five Lemma. This means, in particular, that given a split extension
\begin{equation} \label{split ext Vgrp}
\xymatrix{(X,a)\ar[r]^-n& Z\ar@<-.5ex>[r]_-f&(Y,b)\ar@<-.5ex>[l]_-s}
\end{equation}
in \Grp, where $(X, a)$ and $(Y, b)$ are symmetric $V$-groups, there is at most one $V$-category structure $c$ on $Z$ that turns \eqref{split ext Vgrp} a split extension in $\VGrp_\sym$. As we observed in Proposition \ref{semidir sym V-groups is product}, this structure is always the product structure described in Theorem \ref{t:tensor}.

Moreover, the change-of-base functors of the fibration of points in $\VGrp_\sym$ are not only conservative, but actually monadic. Categories with this property are called \emph{categories with semidirect products}. The reason is that, in such categories, the points (i.e. the split extensions) correspond to suitable \emph{internal actions} (in the sense of \cite{BJK internal actions}). The fact that $\VGrp_\sym$ has semidirect products is a consequence of the following result (of which we recall here a particular case):

\begin{prop}[\cite{MetereMontoli}, Proposition $7$]
Let \C\ be a category with finite limits such that the category \Grp(\C) of internal groups in \C\ has pushouts of split monomorphisms. Then \Grp(\C) has semidirect products.
\end{prop}

Actually, $\VGrp_\sym$ satisfies a stronger categorical-algebraic condition: for any symmetric $V$-group $X$, the functor $\SplExt(-, X)$, associating with every symmetric $V$-group $Y$ the set of isomorphic classes of split extensions in $\VGrp_\sym$ with kernel $X$ and cokernel $Y$, is representable. Taking into account the equivalence between split extensions and internal actions mentioned before, categories with such a property are said to have \emph{representable actions}, or to be \emph{action representative} \cite{BJK action representative}. The fact that $\VGrp_\sym$ has representable actions is a consequence of the following result:

\begin{theorem}[\cite{BJK action representative}, Proposition $1.5$] \label{internal Gp in Cartesian closed cat is action representative}
If $\C$ is a finitely complete cartesian closed category, then the
category $\Grp(\C)$ of internal groups in $\C$ is action
representative.
\end{theorem}

In fact, $\VGrp_\sym$ is the category of internal groups in the finitely complete, cartesian closed category $\VCat$. Following the detailed proof of Theorem \ref{internal Gp in Cartesian closed cat is action representative} (which can be found in \cite{BorCleMon}), we can conclude that the representing object of the functor $\SplExt(-, X)$ is the $V$-group $\Aut(X)$ of maps that are at the same time automorphisms of groups and of $V$-categories, with the $V$-category structure induced by the exponential in $\VCat$. \\

Coming back to the fibration of points, another strong property which holds in $\VGrp_\sym$ is that every change-of-base functor has a right adjoint. This is formalized by saying that $\VGrp_\sym$ is \emph{locally algebraically cartesian closed} (briefly: lacc) \cite{Gray}. Once again, this is a consequence of the fact that $\VGrp_\sym$ is the category of internal groups in $\VCat$:

\begin{prop}[\cite{Gray}, Proposition $5.3$]
Let \C\ be a cartesian closed category with pullbacks. The pullback
functor along any morphism in the category \Grp(\C) of internal groups has a right adjoint.
\end{prop}

This implies, in particular, that $\VGrp_\sym$ is \emph{algebraically coherent} in the sense of \cite{CigoliGrayVdL}.

\section*{Acknowledgements}

The first author was partially supported by the Centre for Mathematics of the University of Coimbra -- UIDB/00324/2020, funded by the Portuguese Government through FCT/MCTES.

This work was partially developed during the first author's visit to the Universit\`{a} degli Studi di Milano, funded by a program for Visiting Professors.

\end{document}